\documentclass[]{class2}

\usepackage{graphicx, xfrac, lineno, float, subcaption, enumitem, xcolor}
\usepackage{comment}
\usepackage[normalem]{ulem}

\begin{document}


\begin{frontmatter}

\titledata{Perfect matchings and Hamiltonicity in the Cartesian product of cycles}{}

\authordata{John Baptist Gauci}
{Department of Mathematics, University of Malta, Malta.}
{john-baptist.gauci@um.edu.mt}
{}

\authordata{Jean Paul Zerafa}
{Dipartimento di Scienze Fisiche, Informatiche e Matematiche,\\ Universit\`{a} di Modena e Reggio Emilia, Via Campi 213/B, Modena, Italy}
{jeanpaul.zerafa@unimore.it}
{}


\begin{abstract}
A pairing of a graph $G$ is a perfect matching of the complete graph having the same vertex set as $G$. If every pairing of $G$ can be extended to a Hamiltonian cycle of the underlying complete graph using only edges from $G$, then $G$ has the PH--property. A somewhat weaker property is the PMH--property, whereby every perfect matching of $G$ can be extended to a Hamiltonian cycle of $G$. In an attempt to characterise all 4--regular graphs having the PH--property, we answer a question made in 2015 by Alahmadi \emph{et al.} by showing that the Cartesian product $C_p\square C_q$ of two cycles on $p$ and $q$ vertices does not have the PMH--property, except for $C_4\square C_4$ which is known to have the PH--property.
\end{abstract}

\end{frontmatter}

\section{Introduction}

All graphs considered are finite, simple (without loops or multiple edges) and connected. A \emph{perfect matching} of a graph $G$ is a set of independent edges of $G$ which cover the vertex set $V(G)$ of $G$. If for a given perfect matching $M$ of $G$ there exists another perfect matching $N$ of $G$ such that $M \cup N$ is a Hamiltonian cycle of $G$, then we say that $M$ can be \emph{extended} to a Hamiltonian cycle. A graph admitting a perfect matching has the \emph{Perfect--Matching--Hamiltonian property} (for short the \emph{PMH--property}) if each of its perfect matchings can be extended to a Hamiltonian cycle. In this case we also say that $G$ is \emph{PMH}. Graphs having this property and other similar concepts have been studied by various authors such as in \cite{AGLMZ19, ThomassenEtAl, AFG, Fink, Haggkvist, LasVergnas, RuskeySavage, WangZhao, Yang}. For a more detailed introduction to the subject we suggest the reader to \cite{AGLMZ19}. 

The \emph{path graph}, \emph{cycle graph} and \emph{complete graph} on $n$ vertices are denoted by $P_n$, $C_n$ and $K_n$, respectively. A vertex of degree one is called an \emph{end vertex}. For any graph $G$, $K_{G}$ denotes the complete graph on the same vertex set $V(G)$ of $G$. Let $G$ be of even order. A perfect matching of $K_{G}$ is said to be a \emph{pairing} of $G$. In \cite{ThomassenEtAl}, the authors say that a graph $G$ has the \emph{Pairing--Hamiltonian property} (for short the \emph{PH--property}) if every pairing $M$ of $G$ can be extended to a Hamiltonian cycle $H$ of $K_{G}$ in which $E(H)-M\subseteq E(G)$. Clearly, this is a stronger property than the PMH--property and if a graph has the PH--property then it is also PMH. Amongst other results, the authors characterise which cubic graphs have the PH--property: $K_{4}$, the complete bipartite graph $K_{3,3}$ and the 3--dimensional hypercube $\mathcal{Q}_{3}$. Most of the notation and terminology that we use in the sequel is standard, and we refer the reader to \cite{BM} for definitions and notation not explicitly stated.

Having a complete characterisation of cubic graphs that have the PH--property, a natural pursuit would be to characterise 4--regular graphs having the same property, as also suggested by the authors in \cite{ThomassenEtAl}. Although Seongmin Ok and Thomas Perrett privately communicated to the authors of \cite{ThomassenEtAl} the existence of an infinite family of 4--regular graphs having the PH--property, it was suggested to tackle this characterisation problem by looking at the Cartesian product of two cycles $C_p\square C_q$ (Open Problem 3 in \cite{ThomassenEtAl}). In particular, the authors ask for which values of $p$ and $q$ does $C_p\square C_q$ have the PH--property.

In this work we show that $C_p\square C_q$ has the PH--property only when both $p$ and $q$ are equal to 4. In fact, the graph $C_4\square C_4$ is isomorphic to the 4--dimensional hypercube $\mathcal{Q}_{4}$, which was proved to have the PH--property in \cite{Fink} together with all other $n$--dimensional hypercubes. More precisely, we show that except for $\mathcal{Q}_{4}$, $C_p\square C_q$ is not PMH.

\section{Main Result}

\begin{definition}The \emph{Cartesian product} $G\square H$ of two graphs $G$ and $H$ is a graph whose vertex set is the Cartesian product $V(G) \times V(H)$ of $V(G)$ and $V(H)$. Two vertices $(u_i,v_j)$ and $(u_k,v_l)$ are adjacent precisely if  $u_i=u_k$ and $v_jv_l\in E(H)$ or $u_iu_k \in E(G)$ and $v_j=v_l$. Thus, $$V(G\square H) = \{(u_r,v_s) : u_r \in V(G) \text{ and } v_s \in V(H)\},\text{ and }$$
$$E(G\square H)=\{(u_i,v_j)(u_k,v_l):u_i=u_k,v_jv_l\in E(H)\text{ or } u_iu_k \in E(G), v_j=v_l\}.$$ 
\end{definition}

For simplicity, we shall refer to the vertex $(u_{r},v_{s})$ as $\omega${\tiny $_{r,s}$}. In this work we restrict our attention to the Cartesian product of a cycle graph and a path graph and to that of two cycle graphs, noting that the latter is also referred to in literature as a torus grid graph. In the sequel we tacitly assume that operations (including addition and subtraction) in the indices of the vertices of a cycle $C_{n}$ are carried out in a ``cyclic sense'', that is, going to $1$ upon reaching $n$, and vice-versa. 

We first prove the following result.

\begin{lemma}\label{LemmaCpPq}
The graph $C_{p}\square P_{q}$ is not PMH, for every $p,q\geq 3$.
\end{lemma}

\begin{proof}
Label the vertices of $C_{p}$ and $P_{q}$ consecutively as $u_{1}, u_{2}, \ldots, u_{p}$ and $v_{1},v_{2}, \ldots, v_{q}$, respectively, such that $v_1$ and $v_q$ are the two end vertices of $P_q$. If $p$ is odd (and so $q$ is even, otherwise $C_p\square P_q$ does not have a perfect matching), then there exists a perfect matching of $C_{p}\square P_q$ containing an odd cut, say $\{\omega${\tiny $_{1,q-1}$}$\omega${\tiny $_{1,q}$}$,\ldots,\omega${\tiny $_{p,q-1}$}$\omega${\tiny $_{p,q}$}$\}$. Clearly, this perfect matching cannot be extended to a Hamiltonian cycle. Thus, we can assume that $p$ is even. Let $M$ be a perfect matching of $C_{p}\square P_q$ containing $\omega${\tiny $_{i,q-1}$}$\omega${\tiny $_{i+1,q-1}$} and $\omega${\tiny $_{i-1,q}$}$\omega${\tiny $_{i,q}$}, for every odd $i\in[p]$, where $[p]=\{1,\ldots, p\}$. For contradiction, suppose that $N$ is a perfect matching of $C_{p}\square P_{q}$ such that $M\cup N$ is a Hamiltonian cycle. Then, for every odd $i\in [p]$, $N$ contains either $\omega${\tiny $_{i,q}$}$\omega${\tiny $_{i+1,q}$} or the two edges $\omega${\tiny $_{i,q-1}$}$\omega${\tiny $_{i,q}$} and $\omega${\tiny $_{i+1,q-1}$}$\omega${\tiny $_{i+1,q}$}. Therefore, $M\cup N$ contains a cycle with vertices belonging to $\{\omega${\tiny $_{1,q-1}$}$,\ldots,\omega${\tiny $_{p,q-1}$}$,\omega${\tiny $_{1,q}$}$,\ldots,\omega${\tiny $_{p,q}$}$\}$. Since $q>2$, $M\cup N$ is not a Hamiltonian cycle, a contradiction. Consequently, $C_{p}\square P_{q}$ is not PMH.
\end{proof}

Now, we prove our main result.

\begin{theorem}
Let $p,q\geq 3$. The graph $C_p \square C_q$ is PMH only when $p=4$ and $q=4$.
\end{theorem}

\begin{proof}
The 4--dimensional hypercube $\mathcal{Q}_4 = C_{4}\square C_{4}$ has the PH--property by Fink's result in \cite{Fink}. Moreover, the authors in \cite{ThomassenEtAl} showed that $C_4\square C_q$ is not PMH when $q\neq 4$. Thus, in what follows we shall assume that $p$ is even and at least 6 and that $q$ is not equal to $4$. Let the consecutive vertices of $C_{p}$ and $C_{q}$ be labelled $u_{1}, u_{2}, \ldots, u_{p}$ and $v_{1},v_{2}, \ldots, v_{q}$, respectively.

We first consider the case when $q=3$. For simplicity, let the vertices $\omega${\tiny $_{i,1}$}$,\omega${\tiny $_{i,2}$}$,\omega${\tiny $_{i,3}$} be referred to as $a_{i}, b_{i}, c_{i}$, for each $i\in[p]$, and let $M$ be a perfect of $C_{p}\square C_{3}$ containing the following nine edges: $a_{1}a_{2},b_{1}b_{2},c_{1}c_{2},a_{3}c_{3},b_{3}b_{4},a_{4}a_{5},c_{4}c_{5},b_{5}b_{6},a_{6}c_{6}$, as shown in Figure \ref{FigureCpC3}. Since $p$ is even, such a perfect matching $M$ clearly exists.

\begin{figure}[h]
      \centering
      \includegraphics[width=0.3\textwidth]{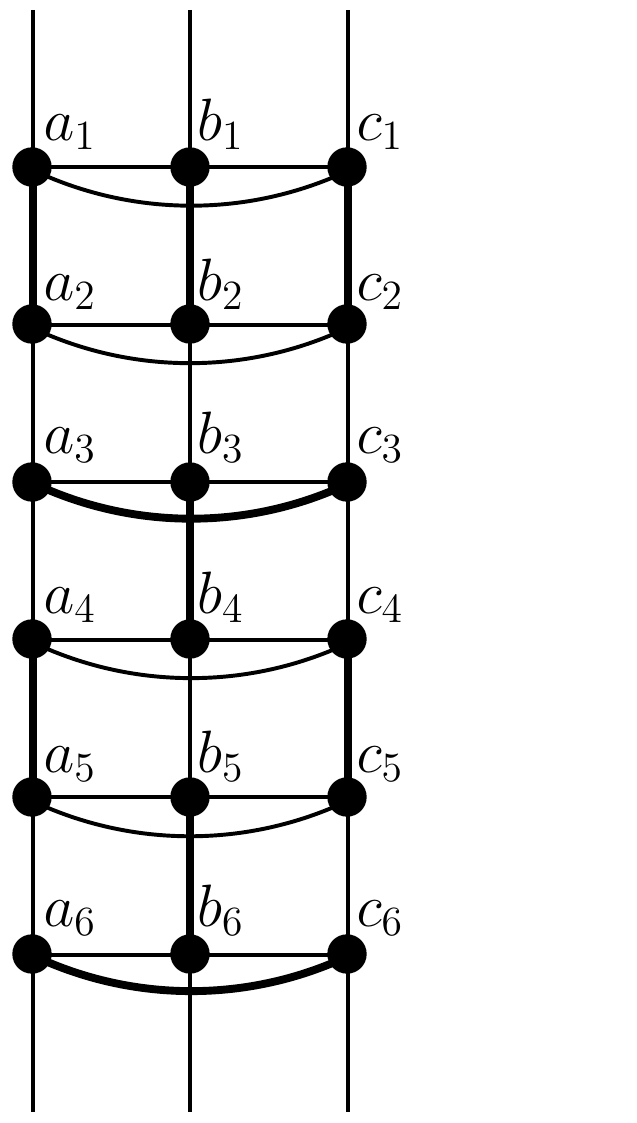}
      \caption{Edges belonging to the perfect matching $M$ in $C_{p}\square C_3$}
      \label{FigureCpC3}
\end{figure}

We claim that $M$ cannot be extended to a Hamiltonian cycle. For, suppose not, and let $N$ be a perfect matching of $C_{p}\square C_3$ such that $M\cup N$ is a Hamiltonian cycle. Each of the two sets $X_{1}=\{a_{3}a_{4},c_{3}c_{4}\}$ and $X_{2}=\{a_{5}a_{6},c_{5}c_{6}\}$ is a 2--edge-cut of the cubic graph $C_{p}\square C_{3}-M$, and so $|X_{i}\cap N|$ is even for each $i=1,2$. Moreover, the edge $b_{4}b_{5}$ is a bridge of the graph $C_{p}\square C_{3}-M$, and consequently, $M\cup N$ contains a cycle of length $4,6$ or $8$ with vertices belonging to $\{a_{3},a_{4},a_{5},a_{6},c_{3},c_{4},c_{5},c_{6}\}$, a contradiction. Therefore, $q\geq 5$.

Similar to above, for each $i\in[p]$, let the vertices $\omega${\tiny $_{i,1}$}$,\omega${\tiny $_{i,2}$}$,\ldots,\omega${\tiny $_{i,6}$} be referred to as $a_{i}, b_{i}, \ldots, f_{i}$ as in Figure \ref{FigureCpCq}, with $f_{i}$ being equal to $a_{i}$ if $q=5$. For each $i\in[p]$, let $L_{i}$ and $R_{i}$ represent $b_{i}c_{i}$ and $d_{i}e_{i}$, respectively, whilst $\mathcal{L}:=\{L_{i}:i\in[p]\}$ and $\mathcal{R}:=\{R_{i}:i\in[p]\}$. Let $M$ be a perfect matching of $C_{p}\square C_{q}$ containing the following edges:
\begin{enumerate}[label=(\roman*)]
\item $a_{i}a_{i+1}$ and $f_{i}f_{i+1}$, for every even $i\in[p]$,
\item $b_{i}b_{i+1}$ and $e_{i}e_{i+1}$, for every odd $i\in[p]$, and
\item $c_{i}d_{i}$, for every $i\in[p]$.
\end{enumerate}

\begin{figure}[h]
      \centering
      \includegraphics[width=0.6\textwidth]{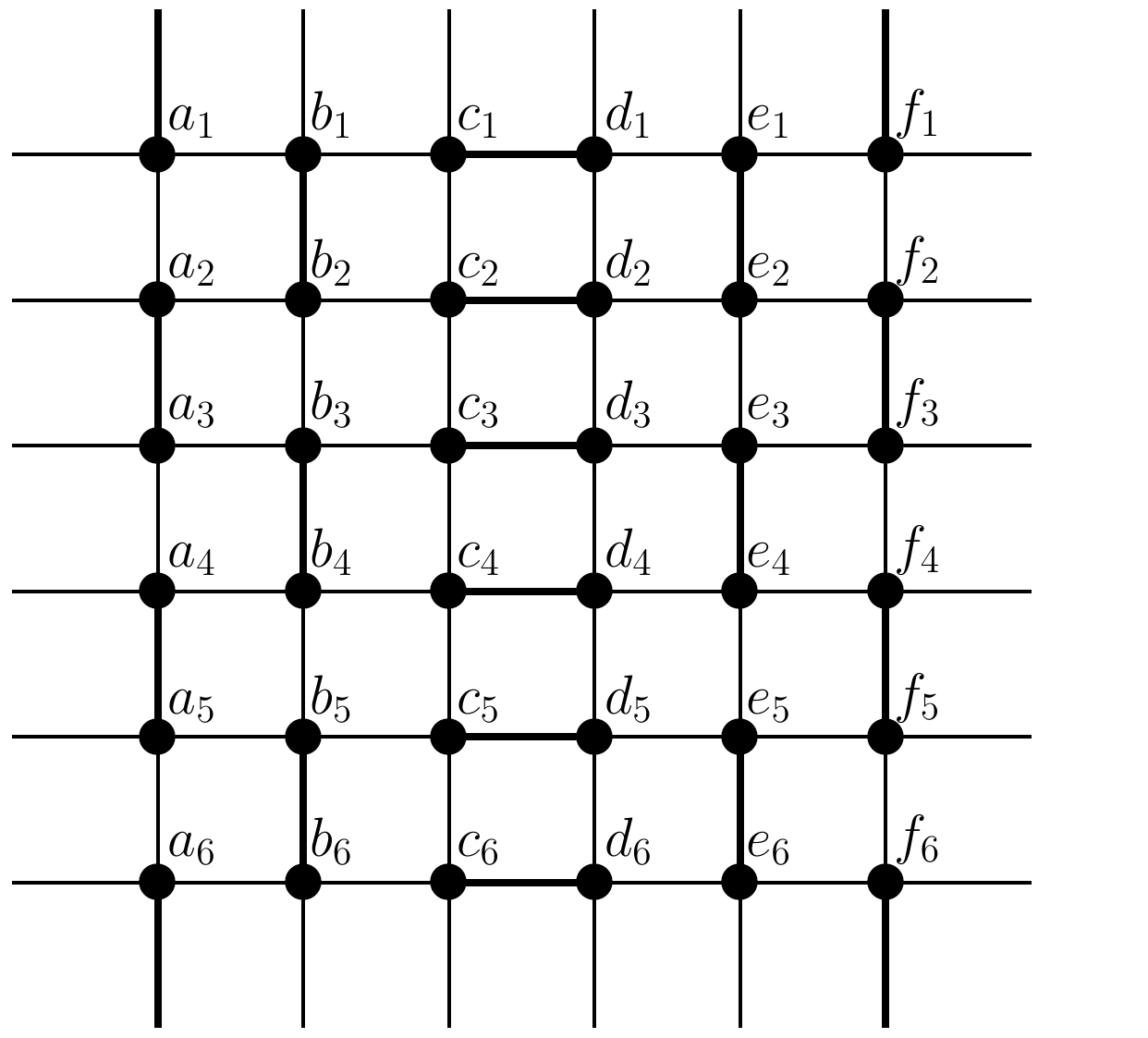}
      \caption{Edges belonging to the perfect matching $M$ in $C_{p}\square C_q$ when $q\geq 5$}
      \label{FigureCpCq}
\end{figure}

Once again, since $p$ is even, such a perfect matching $M$ exists. For contradiction, suppose that $N$ is a perfect matching of $C_{p}\square C_{q}$ such that $M\cup N$ is a Hamiltonian cycle $H$ of $C_{p}\square C_{q}$. The set of edges $\mathcal{L}$ (and similarly $\mathcal{R}$) is an even cut of order $p$ in the cubic graph $C_{p}\square C_{q}-M$. Consequently, both $|\mathcal{L}\cap N|$ and $|\mathcal{R}\cap N|$ are even. We claim that both sets $\mathcal{L}$ and $\mathcal{R}$ must be intersected by $N$. For, suppose that $\mathcal{R}\cap N$ is empty, without loss of generality. In this case, $M\cup N$ forms a Hamiltonian cycle of $C_{p}\square C_{q}-\mathcal{R}$, which is isomorphic to $C_{p}\square P_{q}$. By a similar reasoning to that used in the proof of Lemma \ref{LemmaCpPq}, this leads to a contradiction, and so $M$ cannot be extended to a Hamiltonian cycle. Therefore, both $\mathcal{L}\cap N$ and $\mathcal{R}\cap N$ are non-empty.

Next, we claim that a maximal sequence of consecutive edges belonging to $\mathcal{L}-N$ (or $\mathcal{R}-N$) is of even length, whereby ``consecutive edges" we mean that the indices of these edges are consecutive integers in a cyclic sense. For, suppose there exists such a sequence made up of an odd number of edges. Without loss of generality, let $L_{s}$ and $L_{s+2t}$ be the first and last edges of this sequence, for some $s\in[p]$ and $0\leq t <\sfrac{p}{2}$. Thus, $L_{s-1}$ and $L_{s+2t+1}$ are in $N$. In order for $N$ to cover all the vertices of the graph it must induce a perfect matching of the path $c_{s}c_{s+1}\ldots c_{s+2t}$, which has an odd number of vertices. This is not possible, and so our claim holds. Consequently, there exists $L_{\gamma}\in N$, for some odd $\gamma\in[p]$. We pair the edge $L_{\gamma}$ with the edge $L_{\gamma'}$, where $\gamma'$ is the least integer greater than $\gamma$ in a cyclic sense such that $L_{\gamma'}\in N$. More formally,
\[\gamma'= \left\{
\begin{array}{cl}
\min\{j\in\{\gamma+1,\ldots,p\}: L_{j}\in N\} & $ if such a minimum exists$,\\
  \min\{j\in\{1,\ldots,\gamma-1\}: L_{j}\in N\} & $ otherwise.$
\end{array}\right.\]
By the last claim we note that $\gamma'$ is even and that the next integer $\beta>\gamma'$ in a cyclic sense (if any) for which $L_{\beta}$ is in $N$ must be odd. Repeating this procedure on all the edges in $\mathcal{L}\cap N$ we get a partition of $\mathcal{L}\cap N$ into pairs of edges $\{L_{\gamma},L_{\gamma'}\}$ where $\gamma$ is odd and $\gamma'$ is even. The edges in $\mathcal{R}\cap N$ are partitioned into pairs in a similar way.

We remark that if we start tracing the Hamiltonian cycle $H$ from $c_{\gamma}$ going towards $b_{\gamma}$, then $H$ contains a path with edges alternating in $N$ and $M$, starting from $c_{\gamma}$ and ending at $c_{\gamma'}$. More precisely, if $\gamma'=\gamma+1$, then $H$ contains the path $c_{\gamma}b_{\gamma}b_{\gamma'}c_{\gamma'}$. Otherwise, if $\gamma'\neq \gamma+1$, then, for every even $j\in\{\gamma+1,\ldots,\gamma'-2\}$, $N$ contains either $b_{j}b_{j+1}$ or the two edges $a_jb_j$ and $a_{j+1}b_{j+1}$. Consequently, the internal vertices on this path belong to the set $\{b_{\gamma}, a_{\gamma+1},b_{\gamma+1},\ldots, a_{\gamma'-1},b_{\gamma'-1},b_{\gamma'}\}$. In each of these two cases we shall refer to such a path between $c_{\gamma}$ and $c_{\gamma'}$ as an $L_{\gamma}L_{\gamma'}$\emph{--bracket}, or just a \emph{left--bracket}, with $L_{\gamma}$ and $L_{\gamma'}$ being the \emph{upper} and \emph{lower} edges of the bracket, respectively. 

Having arrived at $c_{\gamma'}$, and noting that $c_{\gamma'}d_{\gamma'}\in M$, $H$ also traverses this edge to arrive at vertex $d_{\gamma'}$. At this point we can potentially take one of three directions, depending on whether $R_{\gamma'}$ is in $N$ or otherwise. If $R_{\gamma'}\in N$, then there exists an $R_{\alpha}R_{\gamma'}$--bracket for some odd $\alpha\in[p]$, where $\alpha$ is the greatest integer smaller than $\gamma'$ in a cyclic sense such that $R_{\alpha}\in N$. As above, this bracket consists of a path with edges alternating in $N$ and $M$, starting from $d_{\gamma'}$ and ending at $d_\alpha$, such that the other vertices of this path belong to:
\[\begin{array}{cl}
\{e_{\gamma'}, f_{\gamma'-1},e_{\gamma'-1},\ldots, f_{\alpha+1},e_{\alpha+1},e_{\alpha}\} & $ if $\alpha\neq \gamma'-1,\\
\{e_{\gamma'},e_{\alpha}\}& $ if $\alpha=\gamma'-1$.$
\end{array}\]
Otherwise, if $R_{\gamma'}\not\in N$, we either have $d_{\gamma'-1}d_{\gamma'}\in N$ or $d_{\gamma'}d_{\gamma'+1}\in N$. Continuing this process, the Hamiltonian cycle $H$ must eventually reach the vertex $c_{\gamma}$. Thus, $H$ contains only vertices in the set $\{a_i,b_i,c_i,d_i,e_i,f_i:i\in [p]\}$, giving a contradiction if $q\geq 7$. Henceforth, we can assume that $5\leq q \leq 6$.
Notwithstanding whether or not $R_{\gamma'}$ is in $N$, if $q=6$, then there is no instance in the above procedure which leads to $H$ passing through the vertices $a_{\gamma}$ and $a_{\gamma'}$, a contradiction. Hence, we can further assume that $q=5$.

We now note that for the vertices in the set $\{a_i,b_i,e_i:i\in [p]\}$ to be in $H$, they must belong either to a left--bracket or to a right--bracket. Thus, if $R_{i}\in N$ is a lower edge of a right--bracket, for some even $i\in[p]$, then, $R_{i+1}$ must be an upper edge of another right--bracket (that is, $R_{i+1}\in N$), otherwise, the vertex $e_{i+1}$ is not contained in any bracket. This observation, together with the fact that a maximal sequence of consecutive edges belonging to $\mathcal{R}-N$ is of even length, implies that if $R_{i}\not\in N$, for some even $i\in [p]$, then $d_{i}d_{i+1}\in N$. 

We revert back to the last remaining case, that is, when $q=5$. The only way how the Hamiltonian cycle $H$ can contain the vertices $a_{\gamma}$ and $a_{\gamma'}$ is when both $R_{\gamma}$ and $R_{\gamma'}$ do not belong to $N$, in which case $a_{\gamma}$ and $a_{\gamma'}$ can be reached by some right--bracket (or right--brackets). Therefore, suppose that $R_{\gamma}$ and $R_{\gamma'}$ do not belong to $N$.

Consequently, tracing $H$ starting from $c_{\gamma}$ and going in the direction of $b_{\gamma}$, after traversing the $L_{\gamma}L_{\gamma'}$--bracket, $H$ must then contain the path $c_{\gamma'}d_{\gamma'}d_{\gamma'+1}c_{\gamma'+1}$. First assume that $\gamma'+1\neq \gamma$. By the same reasoning used for the edges in $\mathcal{R}\cap N$, the lower edge $L_{\gamma'}$ must be followed by an upper edge, and thus $L_{\gamma'+1}\in N$. We trace the Hamiltonian cycle through an $L_{\gamma'+1}L_{\gamma''}$--bracket, noting in particular that for $a_{\gamma''}$ to be in $H$, $R_{\gamma''}$ does not belong to $N$, and hence $d_{\gamma''}d_{\gamma''+1}\in N$, since $\gamma''$ is even. Continuing this procedure, $H$ must eventually reach again the vertex $c_{\gamma}$, without having traversed any right--bracket. The same conclusion can be obtained if $\gamma'+1=\gamma$. In either case, the vertices $a_{\gamma}$ and $a_{\gamma'}$, together with several other vertices of $C_{p}\square C_{q}$, are untouched by $H$, a contradiction. As a result $M$ cannot be extended to a Hamiltonian cycle, proving our theorem.
\end{proof}

\end{document}